\documentclass[reqno]{amsart}

\usepackage{amsmath,amssymb,amsthm,amsfonts,enumerate, enumitem,hyperref,cleveref}
\usepackage{dsfont}
\usepackage[all]{xy}
\usepackage[T1]{fontenc}
\usepackage{tikz}

\DeclareMathOperator{\Aut}{Aut}

\theoremstyle{plain}
\newtheorem{thrm}{Theorem}[section]
\newtheorem{cor}[thrm]{Corollary}
\newtheorem{prop}[thrm]{Proposition}
\newtheorem{lem}[thrm]{Lemma}

\theoremstyle{definition}

\newtheorem{rem}[thrm]{Remark}

\crefrangeformat{equation}{#3(#1)#4--#5(#2)#6}

\crefname{thrm}{Theorem}{Theorems}
\crefname{lem}{Lemma}{Lemmas}
\crefname{cor}{Corollary}{Corollaries}
\crefname{prop}{Proposition}{Propositions}
\crefname{defn}{Definition}{Definitions}
\crefname{exm}{Example}{Examples}
\crefname{rem}{Remark}{Remarks}
\crefname{conj}{Conjecture}{Conjectures}
\crefname{quest}{Question}{Questions}
\crefname{section}{Section}{Sections}
\crefname{equation}{\unskip}{\unskip}
\crefname{enumi}{\unskip}{\unskip}
\crefname{subsection}{Subsection}{Subsections}

\newcommand{\vf}{\varphi}

\newcommand{\sg}{\sigma}

\newcommand{\m}{{}^{-1}}
\newcommand{\sst}{\subseteq}

\newcommand{\lf}{\lfloor}
\newcommand{\rf}{\rfloor}
\newcommand{\wtl}{\widetilde}
\newcommand{\wh}{\widehat}

\newcommand{\spn}{\mathrm{span}}

\newcommand{\id}{\mathrm{id}}
\newcommand{\cI}{\mathcal{I}}

\begin{document}
	\title{Isomorphisms and derivations\\ of partial flag incidence algebras}

	\author{Mykola Khrypchenko}
	\address{Departamento de Matem\'atica, Universidade Federal de Santa Catarina,  Campus Reitor Jo\~ao David Ferreira Lima, Florian\'opolis, SC, CEP: 88040--900, Brazil}
	\email{nskhripchenko@gmail.com}

	
	\begin{abstract}
		Let $P$ and $Q$ be finite posets and $R$ a commutative unital ring. In the case where $R$ is indecomposable, we prove that the $R$-linear isomorphisms between partial flag incidence algebras $I^3(P,R)$ and $I^3(Q,R)$ are exactly those induced by poset isomorphisms between $P$ and $Q$. We also show that the $R$-linear derivations of $I^3(P,R)$ are trivial.
	\end{abstract}
	
	\maketitle
	
	\tableofcontents
	
	\section*{Introduction}

	Recall that the (classical) incidence algebra $I(P,R)$ of a locally finite poset $P$ over a commutative unital ring $R$ is the set of functions $f:P\times P\to R$, such that $f(x,y)=0$ for $x\not\le y$. The elements of $I(P,R)$ thus can be seen as functions on the set of chains of cardinality $n=2$ in $P$, where the product of any two such functions is their convolution. This extends to an arbitrary integer $n\ge 2$ by considering functions on the set of chains of cardinality $n$ (the so-called \textit{``partial $n$-flags''}) with values in $R$ and defining their convolution in a suitable way. Such functions form a ``higher-dimensional'' generalization of $I(P,R)$, called the \textit{$n$-th partial flag incidence algebra} and denoted by $I^n(P,R)$. The algebra $I^n(P,R)$ was introduced by Max Wakefield in~\cite{Wakefield16}, where he used it to define multi-indexed Whitney numbers and then express the coefficients of the Kazhdan-Lusztig polynomial of a matroid~\cite{Elias-Proudfoot-Wakefield16} in terms of these numbers. He showed that $I^n(P,R)$ is not associative and has no (left or right) identity element for $n\ge 3$, unless $P$ is an anti-chain. 
	
	The purpose of this paper is to study some algebraic properties of $I^n(P,R)$, mainly in the case, where $n=3$, $P$ is finite and $R$ is indecomposable. In \cref{sec-prelim} we introduce the elements $e_{\mathbf{x}}\in I^n(P,R)$ which form the natural basis of $I^n(P,R)$ if $|P|<\infty$. We then show in \cref{I^n-not-power-associative} that $I^n(P,R)$ is not even third-power associative, so it fails to belong to many known classes of non-associative algebras. In \cref{sec-ideals} we define the ideals $J^n_k(P,R)$ of $I^n(P,R)$ which will play a crucial role in the description of isomorphisms $I^n(P,R)\to I^n(Q,R)$ in \cref{sec-iso}. We show in \cref{Z^n-is-the-commutator} that $J^3_1(P,R)$ coincides with the commutator submodule of $I^3(P,R)$. We also describe the second and third commutator submodules of $I^3(P,R)$ in terms of $J^3_2(P,R)$ and show in \cref{quotient-Z_2-over-Z_3} that their quotient restores the pairs of elements $x<y$ in $P$, where $y$ covers $x$. This is then used in \cref{sec-iso} to solve the isomorphism problem for $3$-rd partial flag incidence algebras of finite posets over a commutative indecomposable ring (see \cref{I^3(P)-cong-I^3(Q)=>P-cong-Q}). The rest of \cref{sec-iso} is devoted to the proof that the ($R$-linear) isomorphisms between $I^3(P,R)$ and $I^3(Q,R)$ are exactly those induced by poset isomorphisms between $P$ and $Q$ (see \cref{isomorphisms-I^3(P_R)->I^3(Q_R)}). As a consequence, the group $\Aut(I^3(P,R))$ is isomorphic to $\Aut(P)$. In the final \cref{sec-der} of the paper we prove that the derivations of $I^3(P,R)$ are trivial (see \cref{derivations-I^3(P_R)}).
	
	\section{Preliminaries}\label{sec-prelim}
	
	\subsection{General non-associative algebras}
	
	We will be using only the very basic concepts from the theory of non-associative algebras (a good classical source is~\cite{Schafer}). Given a (not necessarily associative) algebra $(A,\cdot)$ over a commutative unital ring $R$ and a non-empty subset $X\sst A$ we denote by $\spn_R (X)$ the $R$-submodule generated by $X$. For any two non-empty subsets $U,V\sst A$, we denote by $UV$ the submodule $\spn_R(\{u\cdot v\mid u\in U,\ v\in V\})$. An \textit{ideal} in $A$ is an $R$-submodule $I\sst A$ such that $AI,IA\sst I$. We write $U^2$ for $UU$. If $I$ is an ideal in $A$, then by the \textit{quotient algebra} $A/I$ we mean the quotient module $A/I$ with the multiplication $(a+I)\cdot(b+I)=a\cdot b+I$. Given $a,b\in A$, we denote by $[a,b]$ the \textit{commutator} $a\cdot b-b\cdot a$ of $a$ and $b$. Thus, $(A,[\phantom{a},\phantom{a}])$ is also an $R$-algebra, and we may define the $R$-submodule $[U,V]=\spn_R(\{[u,v]\mid u\in U,\ v\in V\})$ for any two non-empty subsets $U,V\sst A$. In particular, $[A,A]$ is called the \textit{commutator submodule} of $A$, $[[A,A],[A,A]]$ the \textit{second commutator submodule} of $A$, etc. The direct sum $A\oplus B$ of two $R$-algebras (seen as $R$-modules) is an $R$-algebra under the coordinatewise multiplication. Observe that $A$ and $B$ can be identified with ideals in $A\oplus B$, such that $(A\oplus B)/B\cong A$ and $(A\oplus B)/A\cong B$. A subset $E\sst A$ is a \textit{basis} of $A$ if $A$ is a free $R$-module over $E$. 
	
	An element $e\in A$ is an \textit{idempotent}, if $e^2=e$. Two idempotents $e,f\in A$ are \textit{orthogonal}, if $e\cdot f=f\cdot e=0$. An idempotent $e\in A$ is said to be \textit{primitive} if it cannot be decomposed as $e=e_1+e_2$, where $e_1$ and $e_2$ are orthogonal idempotents, different from $0$ and $e$.
	
	\subsection{Posets}
	
	The terminology we are using here is from~\cite[3.1]{Stanley}. Let $(P,\le)$ be a partially ordered set (poset). A non-empty subset $C\sst X$ is a \textit{chain} if for any two elements $x,y\in C$ either $x\le y$ or $y\le x$. By the length $l(C)$ of a finite chain $C\sst P$ we mean $|C|-1$. The \textit{length} of a finite non-empty poset $P$, denoted by $l(P)$, is the maximum length of chains in $P$. The \textit{interval} between $x$ and $y$ in $P$ with $x\le y$ is the subset $\lf x,y\rf=\{z\in P\mid x\le z\le y\}$. A poset $P$ is said to be \textit{locally finite} if all its intervals are finite. We write $l(x,y)$ for $l(\lf x,y\rf)$.
	
	\subsection{Partial flag incidence algebras}\label{sec-pflinc}
	
	Let $(P,\le)$ be a poset. For any integer $n\ge 2$ denote
	\begin{align*}
		P^n_\le=\{(x_1,\dots,x_n)\in P^n\mid x_1\le\dots\le x_n\}.
	\end{align*}
 	We will also use $P^1_\le:=P$. Each $\mathbf{x}=(x_1,\dots,x_n)\in P^n_\le$, $n\ge 2$, determines
 	\begin{align*}
 		\cI(\mathbf{x})=\lf x_1,x_2\rf\times\dots\times\lf x_{n-1},x_n\rf\sst P^{n-1}_\le.
 	\end{align*}
 	Observe that, for all $n\ge 3$, if $\mathbf{y}\in\cI(\mathbf{x})$, then $\mathbf{z}\in\cI(\mathbf{y})$, where $\mathbf{z}=(x_2,\dots,x_{n-1})$.

 	Let $R$ be a commutative unital ring. The \textit{$n$-th flag incidence algebra~\cite{Wakefield16} of $P$ over $R$} is the set $I^n(P,R)$ of functions $f:P^n_\le\to R$ with the usual $R$-module structure and the following multiplication
 	\begin{align}\label{prod-in-I^n}
 		(fg)(\mathbf{x})=\sum_{\mathbf{y}\in\cI(\mathbf{x})}f(x_1,\mathbf{y})g(\mathbf{y},x_n).
 	\end{align}
 	In particular, $I^2(P,R)$ is the classical incidence algebra~\cite{Rota64} of $P$ over $R$.
 	
	For any $\mathbf{x}=(x_1,\dots,x_n)\in P^n_\le$ we introduce $e_{\mathbf{x}}\in I^n(P,R)$ defined as follows:
	\begin{align}\label{e_x(u)=cases}
		e_{\mathbf{x}}(\mathbf{u})=
		\begin{cases}
			1, &\mathbf{u}=\mathbf{x},\\
			0, &\mbox{otherwise}.
		\end{cases}
	\end{align}
	
	\begin{rem}
		If $P$ is finite, then $\left\{e_{\mathbf{x}}\mid \mathbf{x}\in P^n_\le\right\}$ is a basis of $I^n(P,R)$.
	\end{rem}

	\begin{prop}\label{product-e_x_1...x_n.e_y_1...y_n}
		Let $n\ge 3$ and $\mathbf{x}=(x_1,\mathbf{u})$, $\mathbf{y}=(\mathbf{v},y_n)$ be elements of $P^n_\le$, where $\mathbf{u},\mathbf{v}\in P^{n-1}_\le$. Then
		\begin{align*}
			e_{\mathbf{x}}e_{\mathbf{y}}=
			\begin{cases}
			 \sum_{\mathbf{z}\in\cI(\mathbf{u})} e_{(x_1,\mathbf{z},y_n)}, & \mathbf{u}=\mathbf{v},\\
			 0, & \mathbf{u}\ne\mathbf{v}.
			\end{cases}
		\end{align*}
	\end{prop}
	\begin{proof}
		Fix $\mathbf{p}=(p_1,\dots,p_n)\in P^n_\le$. By \cref{prod-in-I^n}
		\begin{align}\label{sum-e_x(z_1_w)e_y(w_y_n)}
			(e_{\mathbf{x}}e_{\mathbf{y}})(\mathbf{p})=\sum_{\mathbf{w}\in\cI(\mathbf p)}e_{\mathbf{x}}(p_1,\mathbf{w})e_{\mathbf{y}}(\mathbf{w},p_n).
		\end{align}
		If $\mathbf{u}\ne\mathbf{v}$, then either $\mathbf{w}\ne\mathbf{u}$, in which case $e_{\mathbf{x}}(p_1,\mathbf{w})=0$, or $\mathbf{w}\ne\mathbf{v}$, in which case $e_{\mathbf{y}}(\mathbf{w},p_n)=0$, so $(e_{\mathbf{x}}e_{\mathbf{y}})(\mathbf{p})=0$.
			
		Assume now that $\mathbf{u}=\mathbf{v}$. If the sum on the right-hand side of \cref{sum-e_x(z_1_w)e_y(w_y_n)} has a non-zero term, then by \cref{e_x(u)=cases} we have
		$p_1=x_1$, $\mathbf{w}=\mathbf{u}$ and $p_n=y_n$.
 		So, such a term is unique and its value coincides with $e_{(x_1,\mathbf{z},y_n)}(\mathbf{p})$, where $\mathbf{z}=(p_2,\dots,p_{n-1})$. Moreover, it follows from $\mathbf{w}\in\cI(\mathbf p)$ and $\mathbf{w}=\mathbf{u}$ that $\mathbf{z}\in\mathcal{I}(\mathbf{u})$. Conversely, if $e_{(x_1,\mathbf{z},y_n)}(\mathbf{p})\ne 0$ for some $\mathbf{z}\in\mathcal{I}(\mathbf{u})$, then $p_1=x_1$, $(p_2,\dots,p_{n-1})=\mathbf{z}$ and $p_n=y_n$. Hence, $x_1=p_1\le u_1\le p_2\le\dots\le p_{n-1}\le u_{n-1}\le p_n=y_n$, so that $\mathbf{u}\in\cI(\mathbf{p})$. Therefore, the sum on the right-hand side of \cref{sum-e_x(z_1_w)e_y(w_y_n)} has a (unique) non-zero term corresponding to $\mathbf{w}=\mathbf{u}$.
	\end{proof}

	\begin{cor}\label{I^n-not-power-associative}
		Let $n\ge 3$. Then $I^n(P,R)$ is not third power associative, unless $P$ is an antichain.
	\end{cor}
\begin{proof}
	Indeed, if $P$ is not an antichain, then there are $x<y$ in $P$. For any $x\in X$ and $2\le k\le n$ denote
	\begin{align*}
	x^k:=\underbrace{(x,\dots, x)}_k\in P^k_\le.
	\end{align*}
	Take $f=e_{x^n}+e_{(x^{n-1},y)}+e_{(x^{n-2},y^2)}$. Then by \cref{product-e_x_1...x_n.e_y_1...y_n}
	\begin{align*}
	f\cdot f=e_{x^n}+e_{(x^{n-1},y)}+\sum_{x\le z\le y}e_{(x^{n-2},z,y)}
	=e_{x^n}+2e_{(x^{n-1},y)}+\sum_{x<z\le y}e_{(x^{n-2},z,y)}.
	\end{align*}
	Observe that $e_{(x^{n-1},y)}e_{(x^{n-2},z,y)}=0$ unless $z=y$, and $e_{(x^{n-2},y^2)}e_{(x^{n-2},z,y)}=e_{(x^{n-2},z,y)}e_{(x^{n-2},y^2)}=e_{(x^{n-2},z,y)}e_{(x^{n-1},y)}=0$ for all $x<z\le y$. Hence,
	\begin{align*}
	f\cdot (f\cdot f)&=e_{x^n}+2e_{(x^{n-1},y)}+\sum_{x\le z\le y}e_{(x^{n-2},z,y)}=e_{x^n}+3e_{(x^{n-1},y)}+\sum_{x<z\le y}e_{(x^{n-2},z,y)},\\
	(f\cdot f)\cdot f&=e_{x^n}+e_{(x^{n-1},y)}+2\sum_{x\le z\le y}e_{(x^{n-2},z,y)}=e_{x^n}+3e_{(x^{n-1},y)}+2\sum_{x<z\le y}e_{(x^{n-2},z,y)}.
	\end{align*}
	Since $\sum_{x<z\le y}e_{(x^{n-2},z,y)}\ne 0$, we see that $f\cdot (f\cdot f)\ne (f\cdot f)\cdot f$.
\end{proof}

\section{Ideals}\label{sec-ideals}

We introduce, for any integer $k\ge 0$, the $R$-submodules
\begin{align}
J^n_k(P,R)=\{f\in I^n(P,R)\mid f(x_1,\dots,x_n)=0,\mbox{if } l(x_1,x_n)<k\}.
\end{align}
If $l(P)=m<\infty$, then
\begin{align*}
	I^n(P,R)=J^n_0(P,R)\supseteq J^n_1(P,R)\supseteq\dots\supseteq J^n_l(P,R)\supseteq J^n_{m+1}(P,R)=\{0\}.
\end{align*} 

\begin{lem}
	For any $k\ge 0$ the submodule $J^n_k(P,R)$ is an ideal in $I^n(P,R)$.
\end{lem}
\begin{proof}
	Take $f\in J^n_k(P,R)$ and $g\in I^n(P,R)$. Assume that $(fg)(x_1,\dots,x_n)\ne 0$ for some $x_1\le\dots\le x_n$ with $l(x_1,x_n)<k$. Then there are $x_1\le y_1\le x_2\le\dots\le x_{n-1}\le y_{n-1}\le x_n$ such that $f(x_1,y_1,\dots,y_{n-1})g(y_1,\dots,y_{n-1},x_n)\ne 0$. Observe that $l(x_1,y_{n-1})\le l(x_1,x_n)<k$, so $f(x_1,y_1,\dots,y_{n-1})=0$, a contradiction. Similarly, $(gf)(x_1,\dots,x_n)\ne 0$ implies $f(y_1,\dots,y_{n-1},x_n)\ne 0$ for some $x_1\le y_1\le x_2\le\dots\le x_{n-1}\le y_{n-1}\le x_n$ contradicting $f\in J^n_k(P,R)$. 
\end{proof}


\begin{rem}\label{quotient I^n-over-Z^n}
	The quotient algebra $I^n(P,R)/J^n_1(P,R)$ is isomorphic to the direct product $\prod_{x\in P} Re_{(x,\dots, x)}$.
\end{rem}

\begin{rem}
	If $|P|<\infty$, then $J^n_k(P,R)=\spn_R\{e_{(x_1,\dots, x_n)}\mid l(x_1,x_n)\ge k\}$.
\end{rem}
We are going to give another characterization of the ideal $J^n_1(P,R)$. For simplicity, we will consider $n=3$ and $|P|<\infty$, since this is the case we will be dealing with in the rest of the paper. From now on we will write $e_{xyz}$ for $e_{(x,y,z)}\in I^3(P,R)$.

\begin{prop}\label{Z^n-is-the-commutator}
	The ideal $J^3_1(P,R)$ coincides with $[I^3(P,R),I^3(P,R)]$.
\end{prop}
\begin{proof}
	The inclusion $[I^3(P,R),I^3(P,R)]\sst J^3_1(P,R)$ is obvious because $[f,g]\in J^3_1(P,R)$ for all $f,g\in I^3(P,R)$. For the converse inclusion we will prove that $e_{xyz}\in [I^3(P,R),I^3(P,R)]$ for all $x\le y\le z$ with $l(x,z)\ge 1$. Indeed, in this case $x<y$ or $y<z$, so $e_{yyz}e_{xyy}=0$ and $e_{xyz}=e_{xyy}e_{yyz}=[e_{xyy},e_{yyz}]$.
\end{proof}

We now define recursively $Z^n_1(P,R):=J^n_1(P,R)$ and 
\begin{align*}
Z^n_{i+1}(P,R)=[Z^n_i(P,R),Z^n_i(P,R)]
\end{align*}
for $i\in\{1,2\}$.

\begin{lem}\label{Z^3_2(P_R)-basis}
	The submodule $Z^3_2(P,R)$ coincides with
		\begin{align}\label{Z_2-is-span}
		Z^3_1(P,R)^2=\spn_R\left\{e_{xxy}+e_{xyy}\mid l(x,y)=1\right\}\oplus J^3_2(P,R).
		\end{align}
	In particular, $Z^3_2(P,R)$ is a subalgebra of $I^3(P,R)$.
\end{lem}
\begin{proof}
	Denote the right-hand side of \cref{Z_2-is-span} by $V$. If $l(x,y)=1$, then $e_{xxy}+e_{xyy}=e_{xxy}e_{xyy}=[e_{xxy},e_{xyy}]\in Z^3_2(P,R)\cap Z^3_1(P,R)^2$. Let $e_{xyz}\in J^3_2(P,R)$. If $x<y<z$, then $e_{xyz}=e_{xyy}e_{yyz}=[e_{xyy},e_{yyz}]\in Z^3_2(P,R)\cap Z^3_1(P,R)^2$. Let now $x<y=z$. Choose $u\in P$ such that $x<u<z$. Then $[e_{xuz},e_{uzz}]=e_{xuz}e_{uzz}=e_{xzz}+\sum_{u\le v<z}e_{xvz}$. Since $x<u\le v<z$, then $e_{xvz}\in Z^3_2(P,R)\cap Z^3_1(P,R)^2$ as proved above. Therefore, $e_{xyz}=e_{xzz}\in Z^3_2(P,R)\cap Z^3_1(P,R)^2$. Similarly, it follows from $[e_{xxu},e_{xuz}]=e_{xxu}e_{xuz}=e_{xxz}+\sum_{x<v\le u}e_{xvz}$ that $e_{xxz}\in Z^3_2(P,R)\cap Z^3_1(P,R)^2$. Thus, $V\sst Z^3_2(P,R)\cap Z^3_1(P,R)^2$.
	
	Conversely, if $f,g\in Z^3_1(P,R)$ and $l(x,y)=1$, then
	\begin{align*}
	(fg)(x,x,y)=f(x,x,x)g(x,x,y)+f(x,x,y)g(x,y,y)=f(x,x,y)g(x,y,y).
	\end{align*}
	Similarly,
	\begin{align*}
	(fg)(x,y,y)=f(x,x,y)g(x,y,y)+f(x,y,y)g(y,y,y)=f(x,x,y)g(x,y,y).
	\end{align*}
	Hence, $(fg)(x,x,y)=(fg)(x,y,y)$. By symmetry $(gf)(x,x,y)=(gf)(x,y,y)$. This shows that $fg,[f,g]\in V$. Thus, $Z^3_2(P,R)+Z^3_1(P,R)^2\sst V$, proving $Z^3_2(P,R)=Z^3_1(P,R)^2=V$ and thus establishing \cref{Z_2-is-span}.
	
	Since $J^3_2(P,R)$ is an ideal in $I^3(P,R)$, for $Z^3_2(P,R)$ to be a subalgebra, it suffices that $\spn_R\left\{e_{xxy}+e_{xyy}\mid l(x,y)=1\right\}$ be a subalgebra of $I^3(P,R)$. But this is obvious, because $\{e_{xxy}+e_{xyy}\}_{l(x,y)=1}$ are pairwise orthogonal idempotents of $I^3(P,R)$.
\end{proof}

\begin{lem}\label{Z^3_3(P_R)-basis}
	The submodule $Z^3_3(P,R)$ coincides with $J^3_2(P,R)$. In particular, $Z^3_3(P,R)$ is an ideal in $I^3(P,R)$.
\end{lem}
\begin{proof}
	Let $e_{xyz}\in J^3_2(P,R)$. Consider first the case $x<y<z$. We have $e_{xyz}=(e_{xxy}+e_{xyy})(e_{yyz}+e_{yzz})=[e_{xxy}+e_{xyy},e_{yyz}+e_{yzz}]$, where $e_{xxy}+e_{xyy},e_{yyz}+e_{yzz}\in Z^3_2(P,R)$ independently of $l(x,y)$ and $l(y,z)$ by \cref{Z^3_2(P_R)-basis}. Hence, $e_{xyz}\in Z^3_3(P,R)$. Now let $x<y=z$. As in the proof of \cref{Z^3_2(P_R)-basis}, we choose $u\in P$ such that $x<u<z$ and write $e_{xzz}=[e_{xuz},e_{uuz}+e_{uzz}]-\sum_{u\le v<z}e_{xvz}$. Since $e_{xuz},e_{uuz}+e_{uzz}\in Z^3_2(P,R)$ by \cref{Z^3_2(P_R)-basis} and $e_{xvz}\in Z^3_3(P,R)$ for all $u\le v<z$, we conclude that $e_{xyz}=e_{xzz}\in Z^3_3(P,R)$. The case $x=y<z$ is similar. Thus, $J^3_2(P,R)\sst Z^3_3(P,R)$.
	
	Conversely, take $f,g\in Z^3_2(P,R)$ and $x\le y$ such that $l(x,y)=1$. By \cref{Z^3_2(P_R)-basis} we have $f(x,x,y)=f(x,y,y)$ and $g(x,x,y)=g(x,y,y)$. As in the proof of \cref{Z^3_2(P_R)-basis} we calculate
	\begin{align*}
	[f,g](x,x,y)=f(x,x,y)g(x,y,y)-g(x,x,y)f(x,y,y)=0.
	\end{align*}
	Analogously, $[f,g](x,y,y)=0$. Hence, $Z^3_3(P,R)\sst J^3_2(P,R)$.
\end{proof}

\begin{prop}\label{quotient-Z_2-over-Z_3}
	One has the following isomorphism of $R$-algebras:
	\begin{align}\label{iso-Z_2-over-Z_3}
	Z^3_2(P,R)/Z^3_3(P,R)\cong \bigoplus_{l(x,y)=1}R(e_{xxy}+e_{xyy}).
	\end{align}
\end{prop}
\begin{proof}
	In view of \cref{Z^3_3(P_R)-basis,Z^3_2(P_R)-basis} there is an isomorphism of $R$-algebras $Z^3_2(P,R)/Z^3_3(P,R)\cong \spn_R\left\{e_{xxy}+e_{xyy}\mid l(x,y)=1\right\}$. But $\{e_{xxy}+e_{xyy}\}_{l(x,y)=1}$ are pairwise orthogonal idempotents of $I^3(P,R)$, so $\spn_R\left\{e_{xxy}+e_{xyy}\mid l(x,y)=1\right\}$ is isomorphic to $\bigoplus_{l(x,y)=1}R(e_{xxy}+e_{xyy})$ proving \cref{iso-Z_2-over-Z_3}.
\end{proof}

\section{Isomorphisms}\label{sec-iso}

	\subsection{The isomorphism problem}

It is well-known~\cite{St} that the isomorphism problem for usual incidence algebras of locally finite posets over fields has positive solution, i.e., $I(P,R)$ and $I(Q,R)$ are isomorphic if and only if $P$ and $Q$ are isomorphic. This result extends to more general ordered sets~\cite{Belding73,Khripchenko-Novikov09} and coefficient rings~\cite{Voss80,Froelich85,Parmenter-Schmerl-Spiegel90}. We are going to show that the same holds for $3$-rd partial flag incidence algebras of finite posets over indecomposable commutative rings.

Throughout this section we thus assume that $R$ is an indecomposable commutative ring. All the posets are assumed to be finite and all the isomorphisms are $R$-linear. Observe that any isomorphism $\Phi:I^3(P,R)\to I^3(Q,R)$ maps $Z^n_i(P,R)$ to $Z^n_i(Q,R)$, $i=1,2,3$.

	\begin{rem}\label{Phi-maps-Z^n-to-Z^n}
		Let $\Phi:I^3(P,R)\to I^3(Q,R)$ be an isomorphism. Then $\Phi(J^3_1(P,R))=J^3_1(Q,R)$, so $\Phi$ induces an isomorphism $\wtl\Phi:I^3(P,R)/J^3_1(P,R)\to I^3(Q,R)/J^3_1(Q,R)$.
	\end{rem}
	\begin{proof}
		This follows from \cref{Z^n-is-the-commutator}.
	\end{proof}

We introduce the following short notation: $e_x:=e_{xxx}$ for any $x\in P$. Then $\{e_x\mid x\in P\}$ is a set of pairwise orthogonal idempotents of $I^3(P,R)$.

	\begin{cor}\label{Phi-maps-e_x-to-e_vf(x)+rho_x}
		Any isomorphism $\Phi:I^3(P,R)\to I^3(Q,R)$ induces a bijection $\vf:P\to Q$ such that $\wtl\Phi(e_x+J^3_1(P,R))=e_{\vf(x)}+J^3_1(Q,R)$.
	\end{cor}
	\begin{proof}
		Since $R$ is indecomposable, $\{e_x+J^3_1(P,R)\}_{x\in P}$ is the set of all primitive idempotents of $I^3(P,R)/J^3_1(P,R)$ by \cref{quotient I^n-over-Z^n}. Hence,  $\wtl\Phi$ maps bijectively $\{e_x+J^3_1(P,R)\}_{x\in P}$ onto $\{e_y+J^3_1(Q,R)\}_{y\in Q}$.
	\end{proof}

	\begin{cor}\label{Phi(e_xxy+e_xyy)}
		Let $\Phi:I^3(P,R)\to I^3(Q,R)$ be an isomorphism. Then for all $x<y$ with $l(x,y)=1$ there exist $u<v$ with $l(u,v)=1$ and $\sg_{xy}\in J^3_2(Q,R)$ such that $\Phi(e_{xxy}+e_{xyy})=e_{uuv}+e_{uvv}+\sg_{xy}$.
	\end{cor}
	\begin{proof}
		Indeed, $\Phi$ induces an isomorphism $\wtl\Phi$ between the  $Z^3_2(P,R)/Z^3_3(P,R)$ and $Z^3_2(Q,R)/Z^3_3(Q,R)$. Since $R$ is indecomposable,  $\{e_{xxy}+e_{xyy}+Z^3_3(P,R)\}_{l(x,y)=1}$ is the set of all primitive idempotents of $Z^3_2(P,R)/Z^3_3(P,R)$ by \cref{quotient-Z_2-over-Z_3}, so the isomorphism $\wtl\Phi$ maps $e_{xxy}+e_{xyy}+Z^3_3(P,R)$ to $e_{uuv}+e_{uvv}+Z^3_3(Q,R)$ for some $u<v$ with $l(u,v)=1$. Rewriting this in terms of $\Phi$ and using \cref{Z^3_3(P_R)-basis}, we get the desired result.
	\end{proof}
	
	\begin{lem}\label{Phi-maps-e_xxy+e_xyy-to-e_vf(x)vf(x)vf(y)+e_vf(x)vf(y)vf(y)}
		Let $\Phi:I^3(P,R)\to I^3(Q,R)$ be an isomorphism and $\vf:P\to Q$ the corresponding bijection. Then for all $x<y$ with $l(x,y)=1$ we have $\vf(x)<\vf(y)$ and
		\begin{align}\label{Phi(e_xxy+e_xyy)=e_vf(x)vf(x)vf(y)+e_vf(x)vf(y)vf(y)+sg_xy}
			\Phi(e_{xxy}+e_{xyy})=e_{\vf(x)\vf(x)\vf(y)}+e_{\vf(x)\vf(y)\vf(y)}+\sg_{xy},
		\end{align} 
		where $\sg_{xy}\in J^3_2(Q,R)$.
	\end{lem}
	\begin{proof}
		Using \cref{Phi(e_xxy+e_xyy),Phi-maps-e_x-to-e_vf(x)+rho_x}, write $\Phi(e_x)=e_{\vf(x)}+\rho_x$ and $\Phi(e_{xxy}+e_{xyy})=e_{uuv}+e_{uvv}+\sg_{xy}$, where $\rho_x\in J^3_1(Q,R)$ and $\sg_{xy}\in J^3_2(Q,R)$. Since $e_{xxy}=e_x(e_{xxy}+e_{xyy})$, then $\Phi(e_{xxy})=(e_{\vf(x)}+\rho_x)(e_{uuv}+e_{uvv}+\sg_{xy})$. Observe that $\rho_x(e_{uuv}+e_{uvv}+\sg_{xy})\in Z^3_2(Q,R)$ and $e_{\vf(x)}\sg_{xy}\in J^3_2(Q,R)\sst Z^3_2(Q,R)$ by \cref{Z^3_2(P_R)-basis}. Hence, $e_{\vf(x)}(e_{uuv}+e_{uvv})$ cannot be zero, since otherwise $\Phi(e_{xxy})\in Z^3_2(Q,R)$ contradicting \cref{Z^3_2(P_R)-basis}. Thus, $u=\vf(x)$. Similarly, it follows from $e_{xyy}=(e_{xxy}+e_{xyy})e_y$ that $v=\vf(y)$, proving \cref{Phi(e_xxy+e_xyy)=e_vf(x)vf(x)vf(y)+e_vf(x)vf(y)vf(y)+sg_xy}.
	\end{proof}

	\begin{thrm}\label{I^3(P)-cong-I^3(Q)=>P-cong-Q}
		Let $P,Q$ be finite posets and $R$ an indecomposable commutative unital ring. If $I^3(P,R)\cong I^3(Q,R)$, then $P\cong Q$.
	\end{thrm}
	\begin{proof}
		Let $\Phi:I^3(P,R)\to I^3(Q,R)$ be an isomorphism. We will prove that the bijection $\vf:P\to Q$ from \cref{Phi-maps-e_x-to-e_vf(x)+rho_x} is order-preserving. Indeed, if $x<y$ and $l(x,y)=m\ge 1$, then there exist $x=x_0<\dots< x_m=y$ such that $l(x_i,x_{i+1})=1$ for all $0\le i\le m-1$. Applying consecutively \cref{Phi-maps-e_xxy+e_xyy-to-e_vf(x)vf(x)vf(y)+e_vf(x)vf(y)vf(y)} to $x_i<x_{i+1}$, $0\le i\le m-1$, we conclude that $\vf(x)=\vf(x_0)<\dots<\vf(x_m)=\vf(y)$. Since $\vf\m$ is the bijection induced by $\Phi\m$, it is also order-preserving. Thus, $\vf$ is a poset isomorphism between $P$ and $Q$.
	\end{proof}

\subsection{The description of isomorphisms. Base case}

The automorphisms of the incidence algebra of a finite poset over a field were first described in the article by Stanley~\cite{St} mentioned above. The result was later generalized by Baclawski~\cite{Baclawski72} to the case of locally finite posets. He proved that an automorphism of $I(P,R)$ is a composition of an inner automorphism, a multiplicative automorphism and the automorphism induced by an automorphism of $P$. This can be rewritten in terms of a semidirect product of groups as in~\cite{Drozd-Kolesnik07}. There are also generalizations of this description to quasi-ordered sets~\cite{Feinberg76,Scharlau75} and non-locally finite posets~\cite{Kh-aut}. Koppinen~\cite{Kopp} studied homomorphisms of incidence algebras $I(P,A)\to I(P,B)$ and obtained, under certain conditions, similar decompositions. 

Since $I^n(P,R)$ is not unital nor associative for $n\ge 3$, it does not make sense to talk about inner automorphisms of $I^n(P,R)$. One can introduce a generalization of a multiplicative automorphism of a classical incidence algebra to $I^n(P,R)$. However, at least in the case $n=3$, it is easy to show that such an automorphism will be the identity map. Only the automorphisms coming from automorphisms of $P$ non-trivially generalize to $I^3(P,R)$, and we will show that any automorphism of $I^3(P,R)$ is of this form. In fact, we will consider a more general situation of an isomorphism $I^3(P,R)\to I^3(Q,R)$ and prove that it is induced by an isomorphism of posets $P\to Q$.

\begin{prop}
	Let $\vf:P\to Q$ be an isomorphism of finite posets. Then it induces an isomorphism of $R$-algebras $\wh\vf:I^3(P,R)\to I^3(Q,R)$ given by
	\begin{align}
		\wh\vf(e_{xyz})=e_{\vf(x)\vf(y)\vf(z)}
	\end{align}
	for all $x\le y\le z$ in $P$.
\end{prop}
\begin{proof}
	Obvious.
\end{proof}

Given an isomorphism of $R$-algebras $\Phi:I^3(P,R)\to I^3(Q,R)$ and the corresponding isomorphism of posets $\vf:P\to Q$ from \cref{I^3(P)-cong-I^3(Q)=>P-cong-Q}, define $\Psi=(\wh\vf)\m\circ\Phi$. Then $\Psi$ is an automorphism of $I^3(P,R)$ inducing the identity automorphism of $P$. We are going to prove that $\Psi=\id_{I^3(P,R)}$.

	\begin{lem}\label{Phi-maps-e_xxy-to-e_vf(x)vf(x)vf(y)}
		For all $x<y$ with $l(x,y)=1$ there exist $\mu_{xy},\nu_{xy}\in J^3_2(P,R)$ such that
		\begin{align}\label{Phi(e_xxy)-and-Phi(e_xyy)}
		\Psi(e_{xxy})=e_{xxy}+\mu_{xy}\text{ and } \Psi(e_{xyy})=e_{xyy}+\nu_{xy}.
		\end{align} 
	\end{lem}
	\begin{proof}
		Write $\Psi(e_x)=e_{x}+\rho_x$, where $\rho_x\in J^3_1(P,R)$. In view of \cref{Phi-maps-e_xxy+e_xyy-to-e_vf(x)vf(x)vf(y)+e_vf(x)vf(y)vf(y)} we have 
		\begin{align}
			\Psi(e_{xxy})&=\Psi(e_x(e_{xxy}+e_{xyy}))=(e_{x}+\rho_x)(e_{xxy}+e_{xyy}+\sg_{xy})\notag\\
			&=e_{xxy}+r(e_{xxy}+e_{xyy})+\mu_{xy}=(r+1)e_{xxy}+re_{xyy}+\mu_{xy},\label{Phi(e_xxy)=(r+1)e_vf(x)vf(x)vf(y)+re_vf(x)vf(y)vf(y)}
		\end{align}
		where $r=\rho_x(x,x,y)$ and
		\begin{align*}
			\mu_{xy}=\sum_{u<x}\left(\rho_x(u,x,x)e_{uxy}+\rho_x(u,x,y)(e_{uxy}+e_{uyy})\right)+(e_{x}+\rho_x)\sg_{xy},
		\end{align*}
		which belongs to $J^3_2(P,R)$.
		Since, moreover, $0=e_{xxy}e_y$, we obtain by \cref{Phi(e_xxy)=(r+1)e_vf(x)vf(x)vf(y)+re_vf(x)vf(y)vf(y)}
		\begin{align*}
		0=\Psi(e_{xxy})(e_{y}+\rho_y)=((r+1)e_{xxy}+re_{xyy}+\mu_{xy})(e_{y}+\rho_y)=re_{xyy}+\xi_{xy},
		\end{align*}
		where $\xi_{xy}$ is an element of $J^3_2(P,R)$ given by
		\begin{align*}
		\xi_{xy}=\sum_{y<v}\left((r+1)\rho_y(x,y,v)(e_{xxv}+e_{xyv})+r\rho_y(y,y,v)e_{xyv})\right)+\mu_{xy}(e_{y}+\rho_y).
		\end{align*}
		Consequently, $re_{xyy}\in J^3_2(P,R)$, so $r=0$, whence the first equality of \cref{Phi(e_xxy)-and-Phi(e_xyy)}. Similarly, the second equality of \cref{Phi(e_xxy)-and-Phi(e_xyy)} follows from $e_{xyy}=(e_{xxy}+e_{xyy})e_y$ and $0=e_xe_{xyy}$.
	\end{proof}
	
	\begin{cor}\label{e_vf(x)rho_x=0}
		It follows from the proof of \cref{Phi-maps-e_xxy-to-e_vf(x)vf(x)vf(y)} that $\Psi(e_x)(x,x,y)=\Psi(e_y)(x,y,y)=0$ for all $x<y$ with $l(x,y)=1$.
	\end{cor}
	\begin{proof}
		For, $\Psi(e_x)(x,x,y)=\rho_x(x,x,y)$ and $\Psi(e_y)(x,y,y)=\rho_y(x,y,y)$.
	\end{proof}

	\begin{lem}\label{Phi-maps-e_x-to-e_vf(x)}
		For all $x\in P$ there exists $\rho_x\in J^3_2(P,R)$ such that
		\begin{align}\label{Psi(e_x)=e_vf(x)}
		\Psi(e_x)=e_{x}+\rho_x.
		\end{align} 
	\end{lem}
	\begin{proof}
		We know by \cref{Phi-maps-e_x-to-e_vf(x)+rho_x} that $\Psi(e_x)=e_{x}+\rho_x$ for some $\rho_x\in J^3_1(P,R)$. It thus remains to prove that $\rho_x\in J^3_2(P,R)$. By \cref{e_vf(x)rho_x=0} we have $\rho_x(x,x,y)=\Psi(e_x)(x,x,y)=0$ for all $x<y$ with $l(x,y)=1$ and $\rho_x(z,x,x)=\Psi(e_x)(z,x,x)=0$ for all $z<x$ with $l(z,x)=1$. Take $x\ne u<v$ such that $l(u,v)=1$. Since $e_u$ and $e_x$ are orthogonal idempotents, we obtain
		\begin{align*}
			0=\Psi(e_u)\Psi(e_x)=(e_u+\rho_u)(e_x+\rho_x)=e_u\rho_x+\rho_ue_x+\rho_u\rho_x.
		\end{align*}
		Evaluating the values of both sides at $(u,u,v)$ and using $\rho_u(u,u,v)=0$ proved above we have
		\begin{align*}
			0=\rho_x(u,u,v)+\rho_u(u,u,v)\rho_x(u,v,v)=\rho_x(u,u,v).
		\end{align*}
		Similarly it follows from $\Psi(e_x)\Psi(e_v)=0$ evaluated at $(u,v,v)$ that $\rho_x(u,v,v)=0$ for all $u<v\ne x$ with $l(u,v)=1$. This completes the proof of \cref{Psi(e_x)=e_vf(x)}.
	\end{proof}
	
	\begin{prop}\label{Psi(e_xyz)-e_xyz-in-Z_3}
		For all $x\le y\le z$ with $l(x,z)\le 1$ we have $\Psi(e_{xyz})-e_{xyz}\in J^3_2(P,R)$.
	\end{prop}
	\begin{proof}
		A consequence of \cref{Phi-maps-e_xxy-to-e_vf(x)vf(x)vf(y),Phi-maps-e_x-to-e_vf(x)}.
	\end{proof}

\subsection{The description of isomorphisms. Inductive step}

	We are going to generalize \cref{Psi(e_xyz)-e_xyz-in-Z_3} to the case of an arbitrary $l(x,z)$ replacing $J^3_2(P,R)$ by $J^3_i(P,R)$ for an appropriate $i\ge 2$. 
	
	We proceed by induction on $l(x,z)$. Fix $k\ge 1$ and assume that 
	\begin{align*}
		\Psi(e_{xyz})-e_{xyz}\in J^3_{k+1}(P,R)
	\end{align*}
	for all $x\le y\le z$ with $l(x,z)\le k$. We keep the notations $\Psi(e_x)=e_x+\rho_x$, $\Psi(e_{xxy})=e_{xxy}+\mu_{xy}$ and $\Psi(e_{xyy})=e_{xyy}+\nu_{xy}$, where $l(x,y)\le k$. We also write $\Psi(e_{xyz})=e_{xyz}+\eta_{xyz}$ for $x<y<z$ with $l(x,z)\le k$.
	
	The following result will be frequently used without any reference.
	\begin{lem}\label{fg(xyz)-for-l(xz)=j}
		Let $f,g\in J^3_i(P,R)$ and $x\le y\le z$ with $l(x,z)=i$. Then 
		\begin{align}\label{fg(xyz)-one-summand}
			(fg)(x,y,z)=f(x,x,z)g(x,z,z).
		\end{align}
	\end{lem}
	\begin{proof}
		By \cref{prod-in-I^n} we have $(fg)(x,y,z)=\sum f(x,u,v)g(u,v,z)$, where the sum is over all $x\le u\le y\le v\le z$. Observe however that $f(x,u,v)=0$ for $v<z$, because $l(x,v)<l(x,z)=i$ in this case. Similarly, $g(u,v,z)=0$ for $x<u$. Thus, the only summand that can be non-zero corresponds to $u=x$ and $v=z$ giving \cref{fg(xyz)-one-summand}.
	\end{proof}

	\begin{lem}\label{fe_xxy-and-e_xyyf-in-I_j+1}
		Let $f\in J^3_i(P,R)$ and $x\le y\le z$. If $x<y$ then $e_{xyz}f\in J^3_{i+1}(P,R)$, and if $y<z$ then $fe_{xyz}\in J^3_{i+1}(P,R)$.
	\end{lem}
	\begin{proof}
		Since $J^3_i(P,R)$ is an ideal, then $fe_{xyz},e_{xyz}f\in J^3_i(P,R)$. Let $x<y$. Taking $u\le v\le w$ with $l(u,w)=i$ we see that $(e_{xyz}f)(u,v,w)$ can be different from zero only if $x=u<y\le v\le z\le w$, in which case it equals $f(y,z,w)$. But since $l(y,w)<l(x,w)=l(u,w)=i$ and $f\in J^3_i(P,R)$, then $f(y,z,w)=0$. Thus, $(e_{xyz}f)(u,v,w)=0$. Similarly, $(fe_{xyz})(u,v,w)=0$ whenever $y<z$.
	\end{proof}

	\begin{lem}\label{rho=e_xrho+rhoe_x-modulo-I_k+2}
		For any $x\in P$ and $u\le v\le w$ with $l(u,w)=k+1$ one has $\rho_x(u,v,w)=0$ unless $u=v=x$ or $v=w=x$.
	\end{lem}
	\begin{proof}
		Since $\Psi(e_x)=e_x+\rho_x$ is an idempotent, we have
		\begin{align}\label{rho_x=e_xrho_x+rho_xe_x+rho_x^2}
			\rho_x=e_x\rho_x+\rho_xe_x+\rho_x^2.
		\end{align}
		Let $x<u\le v$ such that $l(x,v)=k+1$. Evaluating \cref{rho_x=e_xrho_x+rho_xe_x+rho_x^2} at $(x,x,v)$ and applying \cref{fg(xyz)-for-l(xz)=j} to $\rho_x\in J^3_{k+1}(P,R)$, we obtain $\rho_x(x,x,v)\rho_x(x,v,v)=0$. Now, calculating \cref{rho_x=e_xrho_x+rho_xe_x+rho_x^2} at $(x,u,v)$, we get $\rho_x(x,u,v)=\rho_x(x,x,v)\rho_x(x,v,v)$, whence 
		\begin{align}\label{rho_x(xuv)=0}
			\rho_x(x,u,v)=0.
		\end{align} 
		Similarly, taking $u\le v<x$ with $l(u,x)=k+1$ and calculating \cref{rho_x=e_xrho_x+rho_xe_x+rho_x^2} at $(u,x,x)$ and $(u,v,x)$, we obtain
		\begin{align}\label{rho_x(uvx)=0}
			\rho_x(u,v,x)=0.
		\end{align}
		
		Let now $u\le v\le w$ such that $x\not\in\{u,w\}$ and $l(u,w)=k+1$. Since $\Psi(e_u)\Psi(e_x)=0$, we have
		\begin{align}\label{e_urho_x+rho_ue_x+rho_urho_x=0}
			e_u\rho_x+\rho_ue_x+\rho_u\rho_x=0.
		\end{align}
		Evaluating \cref{e_urho_x+rho_ue_x+rho_urho_x=0} at $(u,u,w)$ we get $\rho_x(u,u,w)+\rho_u(u,u,w)\rho_x(u,w,w)=0$. Now, evaluating the same equality at $(u,w,w)$ we get $\rho_u(u,u,w)\rho_x(u,w,w)=0$, so $\rho_x(u,u,w)=0$. But $\rho_x(u,v,w)=\rho_x(u,u,w)\rho_x(u,w,w)$ by \cref{rho_x=e_xrho_x+rho_xe_x+rho_x^2}. Consequently,
		\begin{align}\label{rho_x(uvw)=0}
			\rho_x(u,v,w)=0.
		\end{align}
		
		Combining \cref{rho_x(xuv)=0,rho_x(uvx)=0,rho_x(uvw)=0}, we come to the desired result.
	\end{proof}
	
	\begin{lem}\label{eta_xyz(xxw)-and-eta_xyz(uzz)}
		Let $x<y<z$ with $l(x,z)\le k$. For any $x<w$ with $l(x,w)=k+1$ one has $\eta_{xyz}(x,x,w)=0$, and for any $u<z$ with $l(u,z)=k+1$ one has $\eta_{xyz}(u,z,z)=0$.
	\end{lem}
	\begin{proof}
		Applying $\Psi$ to $e_xe_{xyz}=0$, we have
		\begin{align}\label{e_xeta_xyz+rho_xe_xyz+rho_xeta_xyz=0}
		e_x\eta_{xyz}+\rho_xe_{xyz}+\rho_x\eta_{xyz}=0.
		\end{align}
		Now, evaluating \cref{e_xeta_xyz+rho_xe_xyz+rho_xeta_xyz=0} at $(x,x,w)$ and using \cref{fg(xyz)-for-l(xz)=j,fe_xxy-and-e_xyyf-in-I_j+1} we obtain $\eta_{xyz}(x,x,w)+\rho_x(x,x,w)\eta_{xyz}(x,w,w)=0$. Furthermore, calculating \cref{e_xeta_xyz+rho_xe_xyz+rho_xeta_xyz=0} at $(x,w,w)$ we get $\rho_x(x,x,w)\eta_{xyz}(x,w,w)=0$. Thus, $\eta_{xyz}(x,x,w)=0$.
		
		Similarly, $e_{xyz}e_z=0$ implies $\eta_{xyz}(u,z,z)=0$.
	\end{proof}

	\begin{lem}\label{mu_xy=e_xmu_xy-and-nu_xy=nu_xye_y-modulo-I_k+2}
		For any $x<y$ with $l(x,y)\le k$ one has $\mu_{xy},\nu_{xy}\in J^3_{k+2}(P,R)$.
	\end{lem}
	\begin{proof}
		By induction hypothesis $\mu_{xy},\nu_{xy}\in J^3_{k+1}(P,R)$. Now take $u\le v\le w$ with $l(u,w)=k+1$. 
		
		\textit{Case 1.} $(u,v)\ne (x,x)$. We are going to prove that $\mu_{xy}(u,v,w)=0$. Applying $\Psi$ to $e_xe_{xxy}=e_{xxy}$, we have
		\begin{align}\label{mu_xy=e_xmu_xy+rho_xe_xxy+rho_xmu_xy}
		\mu_{xy}=e_x\mu_{xy}+\rho_xe_{xxy}+\rho_x\mu_{xy}.
		\end{align}
		Notice that $(\rho_xe_{xxy})(u,v,w)$ is always zero by \cref{fe_xxy-and-e_xyyf-in-I_j+1}, and $(e_x\mu_{xy})(u,v,w)=0$ since $(u,v)\ne (x,x)$. Then evaluating \cref{mu_xy=e_xmu_xy+rho_xe_xxy+rho_xmu_xy} at $(u,v,w)$ and using \cref{fg(xyz)-for-l(xz)=j} we obtain $\mu_{xy}(u,v,w)=\rho_x(u,u,w)\mu_{xy}(u,w,w)$. If $u\ne x$, then $\rho_x(u,u,w)=0$ thanks to \cref{rho=e_xrho+rhoe_x-modulo-I_k+2}, so $\mu_{xy}(u,v,w)=0$. If $u=x$ and $v\ne x$, then $\mu_{xy}(u,v,w)=\rho_x(x,x,w)\mu_{xy}(x,w,w)$. But in this case we calculate \cref{mu_xy=e_xmu_xy+rho_xe_xxy+rho_xmu_xy} at $(x,x,w)$ to see that $\rho_x(x,x,w)\mu_{xy}(x,w,w)=0$. Thus, again $\mu_{xy}(u,v,w)=0$.
		
		\textit{Case 2.} $(v,w)\ne (y,y)$. One similarly gets $\nu_{xy}(u,v,w)=0$ from $e_{xyy}e_y=e_{xyy}$.
		
		\textit{Case 3.} $u=v=x$. Let us now prove that $\mu_{xy}(u,v,w)=0$ in this case. To this end, we apply $\Psi$ to $e_{xxy}e_{xyy}=\sum_{x\le z\le y}e_{xzy}$. We have
		\begin{align}\label{e_xxynu_xy+mu_xye_xyy+mu_xynu_xy=mu_xy+nu_xy+sum.eta_xzy}
			e_{xxy}\nu_{xy}+\mu_{xy}e_{xyy}+\mu_{xy}\nu_{xy}=\mu_{xy}+\nu_{xy}+\sum_{x<z<y}\eta_{xzy}.
		\end{align}
		Notice that $y\ne w$ because $l(x,y)\le k$ while $l(u,w)=k+1$. The element $(e_{xxy}\nu_{xy})(x,x,w)$ can be non-zero only if $y\le w$, in which case it equals $\nu_{xy}(x,y,w)$. But this is zero as proved in Case 2. Now, $(\mu_{xy}e_{xyy})(x,x,w)\ne 0$ implies $w=y$ which is impossible. Thus, the value of the left-hand side of \cref{e_xxynu_xy+mu_xye_xyy+mu_xynu_xy=mu_xy+nu_xy+sum.eta_xzy} at $(x,x,w)$ equals $\mu_{xy}(x,x,w)\nu_{xy}(x,w,w)$. The latter is again zero due to the fact that $y\ne w$, so that $\nu_{xy}(x,w,w)=0$ by the result of Case 2. Now, calculating the right-hand side of \cref{e_xxynu_xy+mu_xye_xyy+mu_xynu_xy=mu_xy+nu_xy+sum.eta_xzy} at $(x,x,w)$, we get $\mu_{xy}(x,x,w)$ since $\eta_{xzy}(x,x,w)=0$ by \cref{eta_xyz(xxw)-and-eta_xyz(uzz)} and $\nu_{xy}(x,x,w)=0$ by Case 2.
		
		\textit{Case 4.} $v=w=y$. As in Case 3, evaluating \cref{e_xxynu_xy+mu_xye_xyy+mu_xynu_xy=mu_xy+nu_xy+sum.eta_xzy} at $(u,y,y)$ one proves that $\nu_{xy}(u,y,y)=0$.
	\end{proof}

	\begin{lem}\label{Psi(e_xyz)-e_xyz-in-J^3_k+2}
		For any $x\le y\le z$ with $0<l(x,z)\le k+1$ one has $\Psi(e_{xyz})-e_{xyz}\in J^3_{k+2}(P,R)$.
	\end{lem}
	\begin{proof}
		\textit{Case 1.} $x<y<z$. Then $l(x,y),l(y,z)\le k$, so by the induction hypo\-the\-sis and \cref{mu_xy=e_xmu_xy-and-nu_xy=nu_xye_y-modulo-I_k+2} we have $\Psi(e_{xyy})=e_{xyy}+\nu_{xy}$ and $\Psi(e_{yyz})=e_{yyz}+\mu_{yz}$, where $\nu_{xy},\mu_{yz}\in J^3_{k+2}(P,R)$. Consequently, $\Psi(e_{xyz})=\Psi(e_{xyy})\Psi(e_{yyz})=(e_{xyy}+\nu_{xy})(e_{yyz}+\mu_{yz})=e_{xyz}+\eta_{xyz}$, where $\eta_{xyz}=e_{xyy}\mu_{yz}+\nu_{xy}e_{yyz}+\nu_{xy}\mu_{yz}$ belongs to $J^3_{k+2}(P,R)$ because $J^3_{k+2}(P,R)$ is an ideal.
		
		\textit{Case 2.} $x=y<z$. Choose an arbitrary $x<u<z$ and write $e_{xxu}e_{xuz}=e_{xxz}+\sum_{x<v\le u}e_{xvz}$. Observe that $l(x,u)\le k$, so $\Psi(e_{xxu})=e_{xxu}+\mu_{xu}$ for some $\mu_{xu}\in J^3_{k+2}(P,R)$ by the induction hypothesis and \cref{mu_xy=e_xmu_xy-and-nu_xy=nu_xye_y-modulo-I_k+2}. Applying also the result of Case 1 to $e_{xuz}$ and $e_{xvz}$, $x<v\le u$, we get
		\begin{align*}
			\Psi(e_{xxz})&=\Psi(e_{xxu})\Psi(e_{xuz})-\sum_{x<v\le u}\Psi(e_{xvz})\\
			&=(e_{xxu}+\mu_{xu})(e_{xuz}+\eta_{xuz})-\sum_{x<v\le u}(e_{xvz}+\eta_{xvz})\\
			&=e_{xxu}e_{xuz}-\sum_{x<v\le u}e_{xvz}+\mu_{xz}=e_{xxz}+\mu_{xz},
		\end{align*}
		where
		\begin{align*}
			\mu_{xz}=e_{xxu}\eta_{xuz}+\mu_{xu}e_{xuz}-\sum_{x<v\le u}\eta_{xvz}\in J^3_{k+2}(P,R).
		\end{align*}
		
		\textit{Case 3.} $x<y=z$. We choose $x<u<z$ and write $e_{xuz}e_{uzz}=e_{xzz}+\sum_{u\le v<z}e_{xvz}$. The rest of the proof is similar to that of Case 2.
	\end{proof}

	\begin{lem}\label{rho-in-I_k+2}
		For any $x\in P$ one has $\Psi(e_x)-e_x\in J^3_{k+2}(P,R)$.
	\end{lem}
	\begin{proof}
		Recall that $\Psi(e_x)=e_x+\rho_x$, where $\rho_x\in J^3_{k+1}(P,R)$ by the induction hypothesis. In view of \cref{rho=e_xrho+rhoe_x-modulo-I_k+2}, to prove that $\rho_x\in J^3_{k+2}(P,R)$, it remains to show that $\rho_x(x,x,y)=\rho_x(z,x,x)=0$ if $l(x,y)=l(z,x)=k+1$. We know by \cref{Psi(e_xyz)-e_xyz-in-J^3_k+2} that $\Psi(e_{xyy})=e_{xyy}+\nu_{xy}$, where $\nu_{xy}\in J^3_{k+2}(P,R)$. Applying $\Psi$ to $e_xe_{xyy}=0$, we have
		\begin{align}\label{e_xnu_xy+rho_xe_xyy+rho_xnu_xy}
			e_x\nu_{xy}+\rho_xe_{xyy}+\rho_x\nu_{xy}=0.
		\end{align}
		Evaluating \cref{e_xnu_xy+rho_xe_xyy+rho_xnu_xy} at $(x,x,y)$ we obtain $\rho_x(x,x,y)=(\rho_xe_{xyy})(x,x,y)=0$, because $e_x\nu_{xy},\rho_x\nu_{xy}\in J^3_{k+2}(P,R)$. Similarly $\rho_x(z,x,x)=0$ follows from $e_{zzx}e_x=0$.
	\end{proof}

\begin{prop}
	The automorphism $\Psi$ is the identity map.
\end{prop}
\begin{proof}
	By induction on $k\ge 1$ with base case \cref{Psi(e_xyz)-e_xyz-in-Z_3} and inductive step \cref{Psi(e_xyz)-e_xyz-in-J^3_k+2,rho-in-I_k+2} one proves that for all $k\ge 1$ and for all $x\le y\le z$ such that $0\le l(x,z)\le k$ one has $\Psi(e_{xyz})-e_{xyz}\in J^3_{k+1}(P,R)$. Taking $k=l(P)$, we obtain the desired result.
\end{proof}

As a consequence, we have the following.
	\begin{thrm}\label{isomorphisms-I^3(P_R)->I^3(Q_R)}
		Let $P,Q$ be finite posets and $R$ an indecomposable commutative unital ring. Then any isomorphism $\Phi:I^3(P,R)\to I^3(Q,R)$ is of the form $\wh\vf$
		for some poset isomorphism $\vf:P\to Q$.
	\end{thrm}

\begin{cor}\label{}
	Let $P$ be a finite poset and $R$ an indecomposable commutative unital ring. Then the group $\Aut(I^3(P,R))$ is isomorphic to $\Aut(P)$.
\end{cor}

	\section{Derivations}\label{sec-der}
	
The description of derivations of $I(P,R)$ is similar to that of automorphisms. Each $R$-linear derivation of $I(P,R)$ is a sum of an inner derivation and the derivation induced by an additive map (see~\cite{Baclawski72}). Dropping the $R$-linearity condition (and maintaining only the additivity), one gets one more class of derivations of $I(P,R)$ as proved in~\cite{Kh-der} (for a more general class of ordered sets). For $I^3(P,R)$ the situation changes drastically. In fact, we will prove that only the zero map is an $R$-linear derivation of $I^3(P,R)$. We assume $P$ to be a finite poset and $R$ an arbitrary commutative ring.

	\begin{lem}\label{D(e_x)-as-a-sum}
		Let $D$ be a derivation of $I^3(P,R)$. Then for all $x\in X$
		\begin{align}\label{D(e_x)=sums_x<y-and-z<x}
		D(e_x)=\sum_{x<v}D(e_x)(x,x,v)e_{xxv}+\sum_{u<x}D(e_x)(u,x,x)e_{uxx}.
		\end{align}
	\end{lem}
	\begin{proof}
		Since $e_x$ is an idempotent, we clearly have $D(e_x)=e_xD(e_x)+D(e_x)e_x$. Therefore,
		\begin{align*}
			D(e_x)&=e_x\left(\sum_{u\le v\le w}D(e_x)(u,v,w)e_{uvw}\right)+\left(\sum_{u\le v\le w}D(e_x)(u,v,w)e_{uvw}\right)e_x\\
			&=\sum_{x\le v}D(e_x)(x,x,v)e_{xxv}+\sum_{u\le x}D(e_x)(u,x,x)e_{uxx}.
		\end{align*}
		Evaluating this at $(x,x,x)$, we see that $D(e_x)(x,x,x)=2D(e_x)(x,x,x)$, whence $D(e_x)(x,x,x)=0$. This proves \cref{D(e_x)=sums_x<y-and-z<x}.
	\end{proof}

\begin{lem}\label{D(e_xy..y)-and-D(e_x..xy)}
	Let $D$ be a derivation of $I^3(P,R)$. Then for all $x<y$
	\begin{align}
	D(e_{xyy})&=\sum_{y<v}D(e_y)(y,y,v)e_{xyv}+\sum_{u\le y}D(e_{xyy})(u,y,y)e_{uyy},\label{D(e_xy..y)=sums_y<u-and-v<y}\\
	D(e_{xxy})&=\sum_{x\le v}D(e_{xxy})(x,x,v)e_{xxv}+\sum_{u<x}D(e_x)(u,x,x)e_{uxy}.\label{D(e_x..xy)=sums_x<u-and-v<x}
	\end{align}
\end{lem}
\begin{proof}
	Since $e_{xyy}=e_{xyy}e_y$, we have $D(e_{xyy})=e_{xyy}D(e_y)+D(e_{xyy})e_y$. Then using \cref{D(e_x)-as-a-sum} we obtain
	\begin{align*}
		D(e_{xyy})&=e_{xyy}\left(\sum_{y<v}D(e_y)(y,y,v)e_{yyv}+\sum_{u<y}D(e_y)(u,y,y)e_{uyy}\right)\\
		&\quad+\left(\sum_{a\le b\le c}D(e_{xyy})(a,b,c)e_{abc}\right)e_y,
	\end{align*}
	which yields \cref{D(e_xy..y)=sums_y<u-and-v<y}. Similarly \cref{D(e_x..xy)=sums_x<u-and-v<x} follows from $e_{xxy}=e_xe_{xxy}$.
\end{proof}

\begin{lem}\label{D-annihilates-e_x}
	Let $D$ be a derivation of $I^3(P,R)$. Then for all $x\in X$ 
	\begin{align}\label{D(e_x)-is-zero}
		D(e_x)=0.
	\end{align}
\end{lem}
\begin{proof}
	Let $x<y$. Then $e_xe_{xyy}=0$. Therefore, $D(e_x)e_{xyy}+e_xD(e_{xyy})=0$. In view of \cref{D(e_x)-as-a-sum,D(e_xy..y)-and-D(e_x..xy)} we have
	\begin{align*}
		0&=\left(\sum_{x<v}D(e_x)(x,x,v)e_{xxv}+\sum_{u<x}D(e_x)(u,x,x)e_{uxx}\right)e_{xyy}\\
		&\quad+e_x\left(\sum_{y<v}D(e_y)(y,y,v)e_{xyv}+\sum_{u\le y}D(e_{xyy})(u,y,y)e_{uyy}\right)\\
		&=D(e_x)(x,x,y)e_{xxy}e_{xyy}=D(e_x)(x,x,y)\sum_{x\le z\le y}e_{xzy}.
	\end{align*}
	It follows that $D(e_x)(x,x,y)=0$. Similarly $e_{zzx}e_x=0$ implies $D(e_x)(z,x,x)=0$ for all $z<x$. Thus, \cref{D(e_x)-is-zero} is clear by \cref{D(e_x)-as-a-sum}.
\end{proof}

\begin{lem}\label{D(e_xzy)-as-a-sum}
	Let $D$ be a derivation of $I^3(P,R)$. Then for all $x<z<y$ 
	\begin{align}\label{D(e_xzy)=sum-D(e_xzz)-and-D(e_zzy)}
	D(e_{xzy})=\sum_{u\le z}D(e_{xzz})(u,z,z)e_{uzy}+\sum_{z\le v}D(e_{zzy})(z,z,v)e_{xzv}.
	\end{align}
\end{lem}
\begin{proof}
	Observe that $e_{xzy}=e_{xzz}e_{zzy}$, so in view of \cref{D-annihilates-e_x,D(e_xy..y)-and-D(e_x..xy)}
	\begin{align*}
		D(e_{xzy})&=D(e_{xzz})e_{zzy}+e_{xzz}D(e_{zzy})\\
		&=\left(\sum_{u\le z}D(e_{xzz})(u,z,z)e_{uzz}\right)e_{zzy}+e_{xzz}\left(\sum_{z\le v}D(e_{zzy})(z,z,v)e_{zzv}\right),
	\end{align*}
	giving \cref{D(e_xzy)=sum-D(e_xzz)-and-D(e_zzy)}.
\end{proof}

\begin{lem}\label{D-annihilates-e_xxy-and-e_xyy}
	Let $D$ be a derivation of $I^3(P,R)$. Then for all $x<y$ 
	\begin{align}\label{D(e_xxy)=D(e_xyy)=zero}
	D(e_{xxy})=D(e_{xyy})=0.
	\end{align}
\end{lem}
\begin{proof}
	Applying $D$ to the product $e_{xxy}e_{xyy}=\sum_{x\le z\le y}e_{xzy}$, we obtain
	\begin{align}\label{D(e_xxy)e_xyy+e_xxyD(e_xyy)=sum-D(e_xzy)}
		D(e_{xxy})e_{xyy}+e_{xxy}D(e_{xyy})=\sum_{x\le z\le y}D(e_{xzy}).
	\end{align}
	By \cref{D-annihilates-e_x,D(e_xy..y)-and-D(e_x..xy)} the left-hand side of \cref{D(e_xxy)e_xyy+e_xxyD(e_xyy)=sum-D(e_xzy)} equals
	\begin{align*}
		&\left(\sum_{x\le v}D(e_{xxy})(x,x,v)e_{xxv}\right)e_{xyy}+e_{xxy}\left(\sum_{u\le y}D(e_{xyy})(u,y,y)e_{uyy}\right)\\
		&\quad=(D(e_{xxy})(x,x,y)+D(e_{xyy})(x,y,y))e_{xxy}e_{xyy}\\
		&\quad=(D(e_{xxy})(x,x,y)+D(e_{xyy})(x,y,y))\sum_{x\le z\le y}e_{xzy}.
	\end{align*}
	Now, by \cref{D(e_xzy)-as-a-sum} the right-hand side of \cref{D(e_xxy)e_xyy+e_xxyD(e_xyy)=sum-D(e_xzy)} is
	\begin{align*}
		\sum_{x\le z\le y}\left(\sum_{u\le z}D(e_{xzz})(u,z,z)e_{uzy}+\sum_{z\le v}D(e_{zzy})(z,z,v)e_{xzv}\right).
	\end{align*}
	Evaluating both sides of \cref{D(e_xxy)e_xyy+e_xxyD(e_xyy)=sum-D(e_xzy)} at $(x,x,y)$ we conclude that
	\begin{align*}
		D(e_{xxy})(x,x,y)+D(e_{xyy})(x,y,y)=D(e_x)(x,x,x)+D(e_{xxy})(x,x,y),
	\end{align*}
	whence $D(e_{xyy})(x,y,y)=D(e_x)(x,x,x)=0$. Similarly, calculating the values of both sides of \cref{D(e_xxy)e_xyy+e_xxyD(e_xyy)=sum-D(e_xzy)} at $(x,y,y)$ we see that
	\begin{align*}
		D(e_{xxy})(x,x,y)+D(e_{xyy})(x,y,y)=D(e_{xyy})(x,y,y)+D(e_y)(y,y,y),
	\end{align*}
	so that $D(e_{xxy})(x,x,y)=0$. Thus, the left-hand side of \cref{D(e_xxy)e_xyy+e_xxyD(e_xyy)=sum-D(e_xzy)} is zero. Now, taking the value of the right-hand side of \cref{D(e_xxy)e_xyy+e_xxyD(e_xyy)=sum-D(e_xzy)} at $(x,x,v)$ with $x\le v\ne y$ we get $D(e_{xxy})(x,x,v)=0$ for all $x\le v\ne y$. Similarly, taking its value at $(u,y,y)$ with $x\ne u\le y$, we obtain $D(e_{xyy})(u,y,y)=0$ for all $x\ne u\le y$. This completes the proof of \cref{D(e_xxy)=D(e_xyy)=zero} in view of \cref{D-annihilates-e_x,D(e_xy..y)-and-D(e_x..xy)}.
\end{proof}

\begin{cor}\label{D-annihilates-e_xzy}
	Let $D$ be a derivation of $I^3(P,R)$. Then for all $x<z<y$ 
	\begin{align}\label{D(e_xzy)=0}
	D(e_{xzy})=0.
	\end{align}
\end{cor}
\begin{proof}
	This follows from \cref{D(e_xzy)-as-a-sum,D-annihilates-e_xxy-and-e_xyy}.
\end{proof}

\begin{thrm}\label{derivations-I^3(P_R)}
	Let $D$ be a derivation of $I^3(P,R)$. Then $D=0$.
\end{thrm}
\begin{proof}
	A consequence of \cref{D-annihilates-e_x,D-annihilates-e_xxy-and-e_xyy,D-annihilates-e_xzy}.
\end{proof}
	\section*{Acknowledgements}
	This work was partially  supported by CNPq 404649/2018-1. The author is grateful to Professor Max Wakefield for fruitful discussions on partial flag incidence algebras, to Professor Ivan Kaygorodov for useful comments on non-associative algebras and to the referee for the suggestion to introduce the notation which helped to improve \cref{sec-pflinc}.

	\bibliography{bibl}{}

\begin{thebibliography}{10}

\bibitem{Baclawski72}
{\sc Baclawski, K.}
\newblock {Automorphisms and derivations of incidence algebras}.
\newblock {\em Proc. Amer. Math. Soc. 36}, 2 (1972), 351--356.

\bibitem{Belding73}
{\sc Belding, W.~R.}
\newblock {Incidence rings of pre-ordered sets}.
\newblock {\em Notre Dame J. Formal Logic 14\/} (1973), 481--509.

\bibitem{Drozd-Kolesnik07}
{\sc Drozd, Y., and Kolesnik, P.}
\newblock {Automorphisms of incidence algebras}.
\newblock {\em Comm. Algebra 35}, 12 (2007), 3851--3854.

\bibitem{Elias-Proudfoot-Wakefield16}
{\sc Elias, B., Proudfoot, N., and Wakefield, M.}
\newblock {The Kazhdan-Lusztig polynomial of a matroid}.
\newblock {\em Adv. Math. 299\/} (2016), 36--70.

\bibitem{Feinberg76}
{\sc Feinberg, R.~B.}
\newblock {Faithful distributive modules over incidence algebras}.
\newblock {\em Pacific J. Math. 65\/} (1976), 35--45.

\bibitem{Froelich85}
{\sc Froelich, J.}
\newblock {The isomorphism problem for incidence rings}.
\newblock {\em Illinois J. Math. 29\/} (1985), 142--152.

\bibitem{Kh-aut}
{\sc Khripchenko, N.~S.}
\newblock {Automorphisms of finitary incidence rings}.
\newblock {\em Algebra and Discrete Math. 9}, 2 (2010), 78--97.

\bibitem{Kh-der}
{\sc Khripchenko, N.~S.}
\newblock {Derivations of finitary incidence rings}.
\newblock {\em Comm. Algebra 40}, 7 (2012), 2503--2522.

\bibitem{Khripchenko-Novikov09}
{\sc Khripchenko, N.~S., and Novikov, B.~V.}
\newblock {Finitary incidence algebras}.
\newblock {\em Comm. Algebra 37}, 5 (2009), 1670--1676.

\bibitem{Kopp}
{\sc Koppinen, M.}
\newblock {Automorphisms and higher derivations of incidence algebras}.
\newblock {\em {J. Algebra} 174}, 2 (1995), 698--723.

\bibitem{Parmenter-Schmerl-Spiegel90}
{\sc Parmenter, M.~M., Schmerl, J., and Spiegel, E.}
\newblock {Isomorphic incidence algebras}.
\newblock {\em Adv. Math. 84}, 2 (1990), 226--236.

\bibitem{Rota64}
{\sc Rota, G.-C.}
\newblock {On the foundations of combinatorial theory. I. Theory of M{\"o}bius
  functions}.
\newblock {\em Z. Wahrscheinlichkeitstheorie und Verw. Gebiete 2}, 4 (1964),
  340--368.

\bibitem{Schafer}
{\sc Schafer, R.~D.}
\newblock {\em {An introduction to nonassociative algebras}}, vol.~22 of {\em
  Pure and Applied Mathematics}.
\newblock Academic Press, New York and London, 1966.

\bibitem{Scharlau75}
{\sc Scharlau, W.}
\newblock {Automorphisms and involutions of incidence algebras}.
\newblock In {\em Represent. Algebr., Proc. int. Conf., Ottawa 1974}, vol.~488
  of {\em Lect. Notes Math}. 1975, pp.~340--350.

\bibitem{St}
{\sc {Stanley}, R.}
\newblock {Structure of incidence algebras and their automorphism groups}.
\newblock {\em {Bull. Am. Math. Soc.} 76\/} (1970), 1236--1239.

\bibitem{Stanley}
{\sc Stanley, R.}
\newblock {\em {Enumerative Combinatorics}}, vol.~1.
\newblock Cambridge University Press, 1997.

\bibitem{Voss80}
{\sc Voss, E.~R.}
\newblock {On the isomorphism problem for incidence rings}.
\newblock {\em {Illinois J. Math.} 24\/} (1980), 624--638.

\bibitem{Wakefield16}
{\sc Wakefield, M.}
\newblock {Partial flag incidence algebras}.
\newblock {\em {\tt (arXiv:1605.01685)}\/} (2016).

\end{thebibliography}
	\bibliographystyle{acm}

\end{document}